\documentclass[11pt, reqno]{amsart}
\usepackage[utf8]{inputenc}
\usepackage{amsmath,bm}
\usepackage{amsthm}
\usepackage{amsfonts}
\usepackage{MnSymbol}
\usepackage{blindtext,amsmath,amsthm,amsfonts, textcomp, mathrsfs, mathtools}
\usepackage[shortlabels]{enumitem}

\usepackage[top=30mm,bottom=30mm,left=30mm,right=30mm]{geometry}
\newtheorem{theorem}{Theorem}[section]

\newtheorem{corollary}{Corollary}

\newtheorem*{theorem*}{Theorem}
\newtheorem*{remark*}{Remark}
\newtheorem*{problem*}{Problem}
\newtheorem*{conjecture*}{Conjecture}
\newtheorem*{question*}{Question}
\newtheorem{lemma}[theorem]{Lemma}

\newcommand{\rom}[1]{\uppercase\expandafter{\romannumeral #1\relax}}

\begin{document}

\title[A uniqueness property of general Dirichlet series]{A uniqueness property of general Dirichlet series}

\author[Anup B. Dixit]{Anup B. Dixit}
\address{Department of Mathematics and Statistics\\ Queen's University\\ Jefferey Hall, 48 University Ave\\ Kingston\\ Canada, ON\\ K7L 3N6}
\email{anup.dixit@queensu.ca}



\maketitle
\date{}

\begin{abstract}
Let $F(s)=\sum_n a_n/\lambda_n^s$ be a general Dirichlet series which is absolutely convergent on $\Re(s)>1$. Assume that $F(s)$ has an analytic continuation and satisfies a growth condition, which gives rise to certain invariants namely the degree $d_F$ and conductor $\alpha_F$. In this paper, we show that there are at most $2d_F$ general Dirichlet series with a given degree $d_F$, conductor $\alpha_F$ and residue $\rho_F$ at $s=1$. As a corollary, we get that elements in the extended Selberg class with positive Dirichlet coefficients are determined by their degree, conductor and the residue at $s=1$.
\end{abstract}

\section{\bf Introduction}
The study of $L$-functions plays a central role in number theory. These are functions attached to arithmetic and geometric objects. Such functions are typically defined as a Dirichlet series of the form
\begin{equation*}
    F(s) = \sum_{n=1}^{\infty} \frac{a_n}{n^s},
\end{equation*}
which is absolutely convergent on a half plane $\Re(s) > \sigma_0$, where the coefficients $a_n$ arise from the underlying object. Then we try to analytically continue $F(s)$ to the whole complex plane. If this is possible, the value distribution of $F(s)$ sheds light on many important arithmetic properties of the underlying structure.\\

The most common example of an $L$-function is the Riemann zeta-function, defined on $\Re(s)>1$ as
\begin{equation*}
    \zeta(s) = \sum_{n=1}^{\infty} \frac{1}{n^s}.
\end{equation*}
It has an analytic continuation to $\mathbb{C}$ except for a simple pole at $s=1$ with residue $1$. Furthermore, it satisfies a functional equation of the following form. If
\begin{equation*}
    \Psi(s) := \pi^{-s/2}\, \Gamma\left(\frac{s}{2}\right)\, \zeta(s),
\end{equation*}
then 
\begin{equation*}
    \Psi(s) = \Psi(1-s).
\end{equation*}
Since $\Gamma(s)$ has poles on all non-positive integers, we get that
\begin{equation}\label{trivial-zeros}
    \zeta(-2n) = 0
\end{equation}
for all $n\in \mathbb{N}$. These are called the trivial zeros of $\zeta(s)$. Moreover, using properties of the gamma function, we get
\begin{equation}\label{growthzeta}
    \max_{|s|=r} |\zeta(s)| \ll \frac{\Gamma(r)}{(2\pi)^r},
\end{equation}
for $r\geq 3/2$. However this growth condition fails if $2\pi$ in \eqref{growthzeta} is replaced by $2\pi + \epsilon$ for any $\epsilon>0$, i.e., for any constant $\eta >0$ there exists an infinite sequence $\{r_m\}$ of positive real numbers going to infinity such that
\begin{equation}\label{notgrowthzeta}
    \max_{|s|=r_m} |\zeta(s)| \geq  \frac{\eta\, \Gamma(r_m)}{(2\pi + \epsilon)^{r_m}}.
\end{equation}

This gives rise to a converse question, which was answered affirmatively by A. Beurling \cite{beurling} in 1950. He proved that
\begin{theorem}[Beurling]\label{beurling-theorem}
Consider a function $F(s)$ satisfying the following properties.
\begin{enumerate}
    \item[(a)] For $0<\lambda_1 < \lambda_2 <\cdots$, let 
\begin{equation*}
    F(s) = \sum_{n=1}^{\infty} \frac{1}{\lambda_n^s}
\end{equation*} 
be absolutely convergent for $\Re(s)>1$.

    \item[(b)] $F(s)$ has an analytic continuation to $\mathbb{C}$ except for a simple pole at $s=1$ with residue $1$.
    
    \item[(c)] $F(-2n)=0$ for all $n\in\mathbb{N}$.
    
    \item[(d)]  $F(s)$ satisfies the growth condition \eqref{growthzeta} and \eqref{notgrowthzeta}.
\end{enumerate}
Then
\begin{equation*}
    F(s) = \zeta(s) \hspace{4mm} \text{or}\hspace{4mm} F(s) = (2^s-1)\, \zeta(s).
\end{equation*}
\end{theorem}

The goal of this paper is to extend this result to any general Dirichlet series of the form 
\begin{equation*}
    F(s) = \sum_{n=1}^{\infty} \frac{a_n}{\lambda_n^s},
\end{equation*}
which is absolutely convergent on $\Re(s)>1$. In order to state the main theorem, we first introduce some growth parameters.

\subsection{\bf Growth Parameters}
Let $F(s)$ be a general Dirichlet series given by
\begin{equation*}
    F(s) = \sum_{n=1}^{\infty} \frac{a_n}{\lambda_n^s},
\end{equation*}
absolutely convergent on $\Re(s)>1$. Suppose $F(s)$ has an analytic continuation to $\mathbb{C}$, except for a simple pole at $s=1$ with residue $\rho_F$. We say that $F(s)$ satisfies the \textit{growth condition} if

\begin{enumerate}
    \item[$(\dagger)$] there exists a positive integer $d_F$ and a positive real number $\alpha_F$ such that
\begin{equation}\label{degree-conductor-growth-condition}
    \max_{s=r} |F(s)| \ll \frac{\Gamma(r)^{d_F}}{\alpha_F^r},
\end{equation}
for $r\geq 3/2$, and for any $\eta,\epsilon > 0$, there exists an infinite sequence $\{r_m\}$ of positive real numbers going to infinity such that
\begin{equation*}
    \max_{s=r_m} |F(s)| \geq \frac{\eta \, \Gamma(r_m)^{d_F}}{(\alpha_F+ \epsilon)^{r_m}}.
\end{equation*}
\end{enumerate}
If $F(s)$ satisfies $(\dagger)$, we call $d_F$ the degree of $F(s)$ and $\alpha_F$ the conductor of $F(s)$. These are closely related to the notion of degree and conductor arising from the functional equation of elements in the Selberg class.\\

In 1989, Selberg \cite{selberg} introduced a class of $L$-functions $\mathbb{S}$ which is expected to encapsulate all familiar $L$-functions arising from arithmetic and geometry. For instance, the Riemann zeta-function $\zeta(s)$, the Dirichlet $L$-functions $L(s,\chi)$, the Dedekind zeta-functions $\zeta_K(s)$ etc. are all members of the Selberg class $\mathbb{S}$. This class has been extensively studied over the past few decades. For precise definition of $\mathbb{S}$ and recent developments, the reader may refer to the excellent survey articles \cite{kaczorowski-survey}, \cite{Rm3}, \cite{Perelli2} and \cite{Perelli1}. For $F\in\mathbb{S}$, there exist real numbers $Q>0$, $\alpha_i\geq 0$, complex numbers $\beta_i$ for $0\leq i \leq k$ and $w\in\mathbb{C}$, with $\Re(\beta_i) \geq 0$ and $|w| =1$, such that
    \begin{equation}\label{fneq}
        \Phi(s) := Q^s \prod_i \Gamma(\alpha_i s + \beta_i) F(s)
    \end{equation}
    satisfies the functional equation
    \begin{equation*}
        \Phi(s) = w \overline{\Phi}(1-\overline{s}).
    \end{equation*}

Although the functional equation is not unique, because of the duplication formula of the $\Gamma$-function, we have some well-defined invariants, namely the degree of $F(s)$ denoted $d_F$, defined as (see \cite{Conrey})
\begin{equation}\label{degree-functional-eqn}
    d_F = 2\sum_{i=1}^m \alpha_i,
\end{equation}
and the conductor of $F(s)$ denoted $q_F$ is defined as (see \cite{Per2})
\begin{equation}\label{conductor-functional-eqn}
    q_F := (2\pi)^{d_F} Q^2 \prod_{i=1}^m \alpha_i^{2\alpha_i}.
\end{equation}

It is an intriguing conjecture that for $F\in\mathbb{S}$, $d_F$ and $q_F$ are always positive integers. It is easily seen that if $F(s)$ satisfies a functional equation of the type \eqref{fneq}, then $F(s)$ also satisfies the growth condition $(\dagger)$. Moreover, the notions of degree $d_F$ in both cases coincide and
\begin{equation*}
    \alpha_F = \frac{(2\pi)^{d_F}}{q_F}.
\end{equation*}

It is worth emphasizing here that the growth condition $(\dagger)$ is a far less restrictive a condition than the functional equation. This is evident from \cite{Km}, where V. K. Murty introduced a class of $L$-functions $\mathbb{M}$ based on growth conditions, which contains the Selberg class $\mathbb{S}$. He proved that $\mathbb{M}$ is closed under addition and has a ring structure, which is not the case for $\mathbb{S}$. A more extensive study of this class is undertaken in \cite{anup-thesis}.

\subsection{\bf The class $\mathbb{B}$}
Let $\mathbb{B}$ denote the class of meromorphic functions $F(s)$ satisfying the following properties.
\begin{enumerate}
    \item {\bf General Dirichlet series} - It can be expressed as a general Dirichlet series
    \begin{equation*}
        F(s) = \sum_{n=1}^{\infty} \frac{a_n}{\lambda_n^s},
    \end{equation*}
    which is absolutely convergent on $\Re(s)>1$, where $0<\lambda_1 < \lambda_2 \cdots $ and $a_n > 0$. We also normalize the leading coefficient, $a_1=1$.
    \item {\bf Analytic continuation} - It has an analytic continuation to $\mathbb{C}$ except for a \textit{simple} pole at $s=1$ with residue $\rho_F$.
    \item {\bf Growth condition} - It satisfies the growth condition $(\dagger)$ with associated invariants $d_F$ and $\alpha_F$.
    \item {\bf Trivial zeros} - $F(-2n d_F)=0$ for all $n\in\mathbb{N}$.
\end{enumerate}
    
%Might want to put some examples here.

In this paper, we consider the question of how many $F\in \mathbb{B}$ can have the same values of degree $d_F$, conductor $\alpha_F$ and residue $\rho_F$ at $s=1$?\\ 

This question is motivated by the classification problem in the Selberg class $\mathbb{S}$. The degree conjecture in $\mathbb{S}$ asserts that for any $F\in\mathbb{S}$, the degree $d_F$ is a non-negative integer. Towards this conjecture, it was shown by Conrey and Ghosh \cite{Conrey} that there are no $F\in \mathbb{S}$ such that $0<d_F<1$. Later Perelli and Kaczorowski \cite{Perelli3} showed that there are no $F\in\mathbb{S}$ such that $1<d_F < 2$. A more intricate question is to classify all elements in $\mathbb{S}$ with a given degree. In this direction, Perelli and Kaczorowski \cite{Perelli4} showed that in the Selberg class if $d_F=1$, then $F(s) = \zeta(s)$ or $F(s) = L(s + i\theta, \chi)$ where $\chi$ is a non-principal irreducible Dirichlet character modulo $q$ and $\theta\in\mathbb{R}$. However, no such classification is known for elements in $\mathbb{S}$ with degree $ 2$ or higher.\\

In this paper, we restrict ourselves to class $\mathbb{B}$ where the functions have positive generalized Dirichlet coefficients and a simple pole at $s=1$. Note that in $\mathbb{B}$, we no longer have the restriction of Euler product or functional equation. Instead we enforce a weaker condition on the growth of the function. Surprisingly, we show that the degree $d_F$, conductor $\alpha_F$ and the residue $\rho_F$, with a certain additional condition determines the function $F(s)$. The proof is inspired by the work of A. Beurling \cite{beurling}.

\subsection{Main Theorem}
\begin{theorem}\label{main-theorem}
Suppose $F\in \mathbb{B}$ satisfies
\begin{equation}\label{main-condition}
    \alpha_F > \left(\frac{\pi \rho_F}{\sin \frac{\pi}{2d_F}}\right)^{d_F}.
\end{equation}
Then there are at most $2d_F$ elements $g\in\mathbb{B}$ such that $d_F = d_g$, $\alpha_F = \alpha_g$ and $\rho_F = \rho_g$.
\end{theorem}
Note that if $F(s)=\zeta(s)$, then Theorem \ref{main-theorem} gives Theorem \ref{beurling-theorem} of Beurling. The \textit{extended Selberg class} $\mathbb{S}^{\#}$ defined in \cite{Per2}, consists of all functions $F(s)$, which have a Dirichlet series representation
\begin{equation*}
F(s)= \sum_{n=1}^{\infty}\frac{a_n}{n^s},
\end{equation*}
on $\Re(s)>1$ with $a_1=1$, can be analytically continued to the whole complex plane $\mathbb{C}$ except for a possible pole at $s=1$ and satisfy a functional equation of the type \eqref{fneq}. In the context of $\mathbb{S}^{\#}$, Theorem \ref{main-theorem} gives

\begin{corollary}
There are at most $2d_F$ elements in the extended Selberg class $\mathbb{S}^{\#}$ having non-negative Dirichlet coefficients , with degree $d_F$, conductor $q_F$ and a simple pole at $s=1$ with residue $\rho_F$, satisfying
\begin{equation*}
    q_F^{-1/d_F} > \frac{ \rho_F}{2\sin \frac{\pi}{2d_F}}.
\end{equation*}
\end{corollary}

\section{\bf Preliminaries}
In this section, we state and prove some lemmas that will be useful in the proof of the Theorem \ref{main-theorem}.

\begin{lemma}\label{derivative-bound}
Suppose $g(s)\in \mathbb{B}$ satisfies the growth condition $(\dagger)$. Then
\begin{equation*}
        \max_{|s| = r} \frac{|g^{(k)}(s)|}{k!} \ll   \frac{ 2^{k+1}\,(\Gamma(r+1))^{d_g}}{\alpha_g^r},
\end{equation*}
    for $r\geq 2$.
\end{lemma}
\begin{proof}
Since $g(s)$ is analytic in the region $|s|\geq 3/2$, using Cauchy's formula, for any $a\in\mathbb{C}$ with $|a|\geq 2$, we have
\begin{equation*}
g^{(k)}(a) = \frac{k!}{2\pi i} \int_{C(a, 1/2)} \frac{g(z)}{(z-a)^{k+1}} \, dz,
\end{equation*} 
where  $C(a,1/2)$ is the circle of radius $1/2$ centered at $a$. Therefore, we have
\begin{align*}
    \frac{|g^{(k)}(a)|}{k!}  \leq \frac{1}{2\pi i} \int_{C(a, 1/2)} \frac{|g(z)|}{(z-a)^{k+1}} \, dz
    \ll \frac{ 2^{k+1}(\Gamma(r+1/2))^{d_g}}{\alpha_g^r}
     \leq \frac{ 2^{k+1}(\Gamma(r+1))^{d_g}}{\alpha_g^r}.
\end{align*}
\end{proof}

We also use the following bound on $\Gamma(s)$.

\begin{lemma}\label{gamma-bound}
For $\sigma \geq 2$,
\begin{equation*}
    \int_{-\infty}^{\infty} |\Gamma(\sigma + 2 + it) | \, dt \ll \sigma^3 \Gamma(\sigma).
\end{equation*}
\end{lemma}
\begin{proof}
Let $s=\sigma+2+it$. For $\sigma \geq 2$ and $t> 2 (\sigma +2)$, by Stirling's approximation, we get
\begin{equation*}
    \Gamma(s) = \sqrt{2\pi}\, s^{s-1/2}\, e^{-s}\, O\left(1 + \frac{1}{|s|}\right).
\end{equation*}
Therefore, we get
\begin{equation}\label{gamma-bound-one}
    |\Gamma(s)| = \sqrt{2\pi}\, |s|^{(\sigma + 2-1/2)}\, e^{-t\, \arg(s)}\, e^{-(\sigma +2)}\,\, O\left(\left |1 + \frac{1}{s}\right|\right).
\end{equation}
Since, $t\geq 2(\sigma + 2)$, we have
\begin{equation*}
 \frac{\pi}{2} > \arg(\sigma + 2 + it) \geq \frac{\pi}{3} \hspace{5mm}  \text{ and } \hspace{5mm}
 |\sigma + 2 + it| \leq \sqrt{5} |t|.
\end{equation*}
Using this in \eqref{gamma-bound-one}, we have
\begin{equation*}
    |\Gamma(\sigma + 2 + it)| \ll t^{(\sigma + 2 -1/2)}\, e^{-t}.
\end{equation*}
Furthermore, since $\Gamma(\overline{z}) = \overline{\Gamma(z)}$, for $s=\sigma +2+ it$, $\sigma\geq 2$ and $t\leq -2(\sigma +2)$, we have
\begin{equation*}
    |\Gamma(\sigma + 2 + it)| \ll |t|^{\sigma + 2 -1/2}\, e^{-|t|}.
\end{equation*}
Thus, we get
\begin{equation*}
    \int_{2(\sigma+2)}^{\infty} |\Gamma(\sigma + 2 + it)| \, dt  \ll \int_{\sigma+2}^{\infty} t^{\sigma + 2 -1/2} e^{-t} \, dt \leq \int_{0}^{\infty} t^{\sigma + 2 -1/2} e^{-t} \, dt  = \Gamma\left(\sigma + 2 - \frac{1}{2}\right) \ll \sigma^3\, \Gamma(\sigma).
\end{equation*}
Similarly, we also have
\begin{equation*}
    \int_{-\infty}^{-2(\sigma+2)} |\Gamma(\sigma + 2 + it)| \, dt \ll \sigma^3\, \Gamma(\sigma).
\end{equation*}
For $|t|< 2(\sigma +2)$, we use the trivial estimate
\begin{equation*}
    |\Gamma(\sigma+2+it)| \leq \Gamma(\sigma+2),
\end{equation*}
to get
\begin{equation*}
    \int_{-2(\sigma+2)}^{2(\sigma + 2)} |\Gamma(\sigma +2 +it)| \ll (\sigma+2) \Gamma(\sigma +2) 
     \ll \sigma^3 \Gamma(\sigma).
\end{equation*}
This completes the proof of the Lemma.
\end{proof}
We also recall the Phragm\'{e}n-Lindel\"of principle (see \cite[section 5.61]{Titchmarsh-functions}).
\begin{theorem}[Phragm\'{e}n-Lindel\"of principle]\label{Phragmen}
Let $f(z)$ be an analytic function of $z=e^{i\theta}$, analytic in the region $D$ between two straight lines making an angle $\pi/\alpha$ at the origin, and on the lines themselves. Suppose that
\begin{equation}\label{phragmen-inequality}
    |f(z)| \leq M
\end{equation}
on the lines, and that, as $r\to\infty$,
\begin{equation*}
    f(z) = O\left( e^{r\beta}\right),
\end{equation*}
where $\beta < \alpha$, uniformly in the angle. Then actually the inequality \eqref{phragmen-inequality} holds throughout the region $D$.
\end{theorem}

A well-known consequence of the Phragm\'{e}n-Lindel\"of theorem is the following theorem due to Carlson (see \cite[section 5.8]{Titchmarsh-functions}), which will also be useful.
\begin{theorem}[Carlson]\label{carlson}
Let $f(z)$ be entire and of the form $O\left(e^{k|z|}\right)$; and let $f(z) = O\left(e^{-a|z|}\right)$, where $a>0$, on the real axis. Then, $f(z)=0$ identically.
\end{theorem}

We also use the following Lemma, which is similar to \cite[Lemma II] {beurling} and the proof follows a similar argument.

\begin{lemma}\label{uniqueness-growth}
There are only two entire functions $f$ of exponential type with $f(0)=1$ which on the real axis satisfy a relation of the form
\begin{equation}\label{uniqueness-hrowth-condition}
    f(x) - a |x|^p e^{\beta |x|}  = O(e^{-\delta |x|}),
\end{equation}
where $a>0, \beta>0, \delta >0$ and $p$ are real constants, viz.:
\begin{equation*}
    f(z) = \frac{e^{\beta z} - e^{-\beta z}}{2\beta z} 
\end{equation*}
and
\begin{equation*}
    f(z) = \frac{e^{\beta z} + e^{-\beta z}}{2}.
\end{equation*}

\end{lemma}

\begin{proof}
Since $f(x)$ satisfies \eqref{uniqueness-hrowth-condition} on the real line, we have 
\begin{equation*}
    f(x)-f(-x) = O(e^{-\delta|x|})
\end{equation*}
on the real line. Therefore, by Carlson's theorem \ref{carlson}, we conclude that $f$ is an even function. Since, $f(z)$ is of exponential type, there exists $A>0$ such that 
\begin{equation*}
    \log |f(z)| < A |z|
\end{equation*}
for $|z|>1$. Thus, on the boundary of the region $\Re(z)>1$, $\Im(z)>0$, the function $h(z) := f(z) - a z^p e^{\beta z}$ satisfies
\begin{equation*}
    |h(z)| \ll e^{(iA-\delta) z}.
\end{equation*}
By applying Phragm\'{e}n-Lindel\"of Theorem \ref{Phragmen}, we get that for $z=x+iy$ 
\begin{equation}\label{boundary-inequality-1}
    |h(z)| \ll e^{Ay-\delta x}
\end{equation}
holds for the whole region $\Re(z)>1$, $\Im(z)>0$. By a similar argument, we also have the bound \eqref{boundary-inequality-1} for the region $\Re(z)>1$, $\Im(z)<0$. Hence, in the strip, $|\Im(z)|<1$ and $\Re(z)>1$, we have $h(z) = O(e^{-\delta x})$. Thus, using Cauchy's formula,
\begin{equation*}
    h^{(n)}(x) = \frac{n!}{2\pi i} \int_{C(x,1)} \frac{h(z)}{(z-x)^{n+1}} \, dz,
\end{equation*}
where $C(x,1)$ is a circle of radius $1$ with center $x$, we get
\begin{equation*}
    h^{(n)}(x) = O(e^{-\delta x})
\end{equation*}
for all real $x>1$. Since $h$ is even, we further get
\begin{equation*}
    h^{(n)}(x) = O(e^{-\delta |x|})
\end{equation*}
for all real $x$. Now, note that for real $x\neq 0$,  $|x|^p e^{\beta x}$ and $|x|^p e^{-\beta x}$ are both solutions to the linear differential equation
\begin{equation*}
    L(u) = x^2 u'' - 2 p x u' - (\beta^2 x^2 - p(p+1)) u =0
\end{equation*}
Therefore, for real $x\neq 0$, we have
\begin{equation*}
    L(f) = L(h) = O\left(x^2 e^{-\delta |x|}\right).
\end{equation*}
Since $L(f)$ is itself an entire function of exponential type, by Carlson's theorem \ref{carlson}, $L(f)$ must vanish identically. Thus, for $x>0$, $f(x)$ is of the form
\begin{equation*}
    f(x) = x^p (c_1 e^{\beta x} + c_2 e^{-\beta x}).
\end{equation*}
Since $f$ is entire, even and satisfies $f(0)=1$, the only choices we have are
\begin{equation*}
    p=-1, c_1 = -c_2 = \frac{1}{2\beta},
\end{equation*}
or
\begin{equation*}
    p=0, c_1 = c_2 = \frac{1}{2}.
\end{equation*}
This proves the lemma.
\end{proof}

\section{\bf Proof of Theorem \ref{main-theorem}}
Let $F(s)\in \mathbb{B}$ be fixed. Suppose $g(s) \in \mathbb{B}$ satisfies $d_g = d_F$, $\alpha_g = \alpha_F$ and $\rho_g=\rho_F$. Our goal is to show that there are at most $2 d_F$ possibilities for $g(s)$. Let the general Dirichlet series for $g(s)$ be given by
\begin{equation*}
    g(s) = \sum_{n=1}^{\infty} \frac{a_n}{\lambda_n^s},
\end{equation*}
which is absolutely convergent for $\Re(s)>1$. Thus, there exists a $\delta>0$ such that $\lambda_n/n > \delta$ for all $n$. Now consider the function
\begin{equation*}
    f(z) := \prod_{n=1}^{\infty} \left( 1 + \frac{z^2}{\lambda_n^{2d_F}}\right)^{a_n}.
\end{equation*}
Since $\lambda_n/n >\delta >0$ and $a_n \geq 0$ for all $n$, $f(z)$ is an entire function of exponential type. Using the well-known identity (see \cite[section 1.4, formula 4.4]{tables_mellin_transform})
\begin{equation*}
    \int_0^{\infty} \log (1+x^2) \, \frac{dx}{x^{1+s}} = \frac{\pi}{s \sin \frac{\pi s}{2}}
\end{equation*}
for $0<\Re(s)<2$, we have

\begin{equation*}
   \int_0^{\infty} \log (1+x^2) \, \frac{dx}{x^{1+s/d_F}} = \frac{\pi\, d_F}{s \sin \frac{\pi s}{2d_F}} 
\end{equation*}
for $0< \Re(s) < 2d_F$. Changing variables $x \mapsto x/\lambda_n^{d_F}$, we get
\begin{equation*}
    \int_0^{\infty} \log \left(1+\frac{x^2}{\lambda_n^{2d_F}}\right)^{a_n} \, \frac{dx}{x^{1+s/d_F}} = \frac{a_n}{\lambda_n^s}\left(\frac{\pi \, d_F}{s \sin \frac{\pi s}{2d_F}}\right)
\end{equation*}
for $0<\Re(s)<2d_F$. Summing over $n$, we have
\begin{equation*}
    \int_0^{\infty}\log f(x)\, \frac{dx}{x^{1+s/d_F}} = \left(\frac{\pi d_F}{s \sin \frac{\pi s}{2d_F}}\right) g(s),
\end{equation*}
for $1<\Re(s)<2d_F$. Now define 
\begin{equation*}
    \psi(s) := \left(\frac{\pi d_F}{s \sin \frac{\pi s}{2d_F}}\right) g(s).
\end{equation*}
Since $g(s)$ has an analytic continuation to $\mathbb{C}$ except for a simple pole at $s=1$, we have a meromorphic continuation for $\psi(s)$ in the region $\Re(s) < 2d_F$. Furthermore, since $g(-2nd_F)=0$ for all $n$, the zeros of $\sin (\pi s/2d_F)$ do not generate poles for $\psi(s)$ in this region. Therefore, $\psi(s)$ has an analytic continuation on $\Re(s)< 2d_F$, except for poles at $s=0$ and $s=1$, with principal parts
\begin{equation*}
    \frac{2\,d_F^2 \,g(0)}{s^2} + \frac{2\,d_F^2 \,g'(0)}{s}\hspace{5mm}\text{and}\hspace{5mm} \frac{1}{(s-1)} \left(\frac{\pi \rho_F\, d_F }{\sin \frac{\pi}{2d_F}}\right)
\end{equation*}
respectively.\\ 

We now capture the growth of $\psi(s)$. Observe that on the boundary of the region $\Im(s)>0$, $\Re(s)< 3/2$ and $|s|>2$,
\begin{equation}\label{bound-for-psi}
    \left | \frac{\psi(s)}{\alpha_F^s \, \Gamma(2-s)^{d_F}} \right| \ll 1.
\end{equation}
Indeed, on the vertical line $s= 3/2 + it$ and $|s|>2$, we have
$|g(s)|\ll 1$. We also have
\begin{equation*}
    \left |\Gamma\left(\frac{1}{2} +it \right)\right| > \sqrt{\pi} e^{-\frac{\pi}{2} |t|} \hspace{5mm}\text{ and }\hspace{5mm} \left|\sin \left(\frac{\pi}{2d_F}\left(\frac{3}{2} + it\right)\right)\right| > \frac{1}{4} e^{-\frac{\pi}{2d_F}|t|}.
\end{equation*}
Therefore, we get the bound for $s= 3/2 + it$ and $|s|>2$.\\ 

On the negative real axis, we consider the following two cases. If $|\sigma + 2nd_F| \leq  \frac{1}{4d_F}$ for $n\in \mathbb{N}$, we have $(\sigma + 2nd_F)/\sin (\pi \sigma/2d_F)$ is bounded above and below by positive constants. Using the Taylor expansion of $g(s)$ at $s=-2nd_F$ and Lemma \ref{derivative-bound}, we get
\begin{equation*}
    \left| \frac{g(\sigma)}{\sin \frac{\pi \sigma}{2d_F}} \right| \ll \left|(\sigma+2nd_F) g'(-2nd_F) + (\sigma+2nd_F)^2\, \frac{g''(-2nd_F)}{2} + \cdots \right| 
     \ll \left|\frac{(\Gamma(r))^{d_F}}{(k+\epsilon)^r}\right |,
\end{equation*}
for every $\epsilon>0$. On the other hand, if $|\sigma + 2nd_F| >  \frac{1}{4d_F}$, we have 
\begin{equation*}
    \left| \frac{1}{\sin (\pi \sigma/2d_F)}\right| \ll 1.
\end{equation*}
Now, using the Euler's reflection formula 
\begin{equation*}
    \Gamma(z)\Gamma(1-z) = \frac{\pi}{\sin \pi z}
\end{equation*}
for $z\notin \mathbb{Z}$ and the growth condition $(\dagger)$, we get for $\sigma<0$,
\begin{align*}
    \left | \frac{\psi(\sigma)}{\alpha^{\sigma} \Gamma(2- \sigma)^{d_F}} \right|  = \left | \frac{g(\sigma)}{\alpha^{\sigma} \sin (\frac{\pi \sigma}{2}) \Gamma(2- \sigma)^{d_F}} \right| 
     \ll \left | \frac{(\Gamma(-\sigma))^{d_F}}{|\sigma|\alpha^{\sigma}  \alpha^{-\sigma} (\Gamma(2- \sigma))^{d_F}} \right|
     \ll 1.
\end{align*}
This proves the bound \eqref{bound-for-psi}. Applying Phragm\'{e}n-Lindel\"of principle, we get that all $s$ in the region $\Im(s)>0$, $\Re(s)<3/2$ and $|s|>2$ satisfy
\begin{equation}\label{boun-psi-region}
    \left | \frac{\psi(s)}{\alpha_F^s \, \Gamma(2-s)^{d_F}} \right| \ll 1.
\end{equation}
By symmetry, we in fact have \eqref{boun-psi-region} in the region $\Re(s)<3/2$ and $|s|>2$.\\

Now, using Lemma \ref{gamma-bound}, we have for $\sigma \geq 2$,
\begin{equation}\label{psi-bound}
    \int_{-\infty}^{\infty} |\psi(-\sigma + it)| \, dt \ll \frac{\sigma^{3d_F} (\Gamma(\sigma))^{d_F}}{\alpha^{\sigma}}.
\end{equation}
Applying Mellin transform, for $d_F <c< 2d_F$
\begin{equation*}
    \log f(x) = \frac{1}{2\pi i d_F} \int_{(c)} \psi(s) x^{s/d_F} \, ds.
\end{equation*}
Moving the line of integration to the left, we get for $c<0$,

\begin{equation*}
     \frac{1}{2\pi i d_F} \int_{(c)} \psi(s) x^{s/d_F} \, ds = \log f(x) - \frac{\pi \,\rho_F\, d_F}{\sin \frac{\pi}{2d_F}} x^{1/d_F} - 2 d_F\, g(0) \log x - 2 d_F^2 \,g'(0).
\end{equation*}
Set $a=2\, d_F \,g(0)$ and $b= e^{2 \,d_F^2 \, g'(0)}$ and
\begin{equation*}
    m = \frac{\pi \, \rho_F \, d_F}{\sin \frac{\pi}{2d_F}}.
\end{equation*} 
Using \eqref{psi-bound}, for $c>2$
\begin{equation*}
    \left| \log f(x) - \log (b x^a e^{m x^{1/d_F}}) \right| \ll \frac{c^{d_F} (\Gamma(c))^{3d_F}}{(\alpha_F x)^c}.
\end{equation*}
Choose $c = (\alpha_F \, x)^{1/d_F}$. Using Stirlings formula and taking $x\to \infty$, we get
\begin{equation*}
    \frac{c^{d_F} (\Gamma(c))^{d_F}}{(\alpha_F x)^c} \ll e^{-((\alpha_F -\epsilon) x)^{1/d_F} d_F },
\end{equation*}
for every $\epsilon>0$. Thus
\begin{equation*}
     \left| \log f(x) - \log (bx^a e^{\pi \rho x}) \right| \ll  e^{-((\alpha_F-\epsilon) x)^{1/d_F} d_F },
\end{equation*}
for every $\epsilon>0$. If $f(x)>bx^a e^{\pi \rho x}$, exponentiating both sides gives
\begin{equation*}
    \frac{f(x)}{b x^a e^{mx^{1/d_F}}} = 1 + O\left( e^{-((\alpha_F-\epsilon) x)^{1/d_F} d_F }\right)
\end{equation*}
Using the condition \eqref{main-condition} and the fact $e^y = 1 + O(y)$ for $y\ll 1$, we have for $x\to\infty$
\begin{equation}\label{equation-1}
    f(x) = b x^a e^{mx^{1/d_F}} + O\left( e^{-\delta x^{1/d_F}}\right).
\end{equation}
Moreover, if $f(x) < bx^a e^{\pi \rho x}$, using the condition \eqref{main-condition}, we also get \eqref{equation-1}. Since $f$ is even,
\begin{equation*}
    \left| f(x) - b |x|^a e^{m|x|^{1/d_F}}\right| \ll e^{-\delta |x|^{1/d_F}}
\end{equation*}
for all real $x \neq 0$. From the definition of $f(z)$, note that $h(z) := f(z^{d_F})$ is also an entire function of exponential type and satisfies
\begin{equation*}
     \left| h(x) - b |x|^{a/d_F} e^{m|x|}\right| \ll e^{-\delta |x|}
\end{equation*}
for all real $x\neq 0$. Using Lemma \ref{uniqueness-growth} on $h(z)$, we conclude that there are at most two such functions. Hence, there are at most $2 d_F$ choices for $f(x)$ and therefore for $g(s)$. This proves Theorem \ref{main-theorem}.

\section{\bf Application to Dedekind zeta-functions}
For a number field $K/\mathbb{Q}$, let $n_K$ denote the degree $[K:\mathbb{Q}]$ and $|\Delta_K|$ denote the absolute discriminant. The Dedekind zeta-function associated to $K$ is defined as
\begin{equation*}
\zeta_K(s):= \prod_{\mathfrak{P}\subset \mathcal{O}_K} \left( 1- N\mathfrak{P}^{-s}\right)^{-1},
\end{equation*}
for $\Re(s)>1$, where $\mathfrak{P}$ runs over all non-zero prime ideals in the ring of integers of $K$. The function $\zeta_K(s)$ has an analytic continuation to the whole complex plane except for a simple pole at $s=1$. Let the residue of $\zeta_K(s)$ at $s=1$ be $\rho_K$. Additionally, $\zeta_K(s)$ satisfies a functional equation and hence, the growth condition $(\dagger)$ with $d_{\zeta_K} = n_K$ and $\alpha_{\zeta_K} = (2\pi)^{n_K} / |\Delta_K|$. Applying Theorem \ref{main-theorem}, we have
\begin{corollary}\label{Dedekind-zeta-corollary}
For the same values of $n_K$, $|\Delta_K|$ and $\rho_K$, there are at most $2n_K$ Dedekind zeta-functions $\zeta_K(s)$ satisfying
\begin{equation}\label{dedekind-zeta-condition}
    \left|\Delta_K^{1/n_K}\right| < \frac{2 \sin (\pi/2n_K)}{\rho_K}.
\end{equation}
\end{corollary}

For a fixed degree $n_K$, as we vary the number fields $K$, the root discriminant $|\Delta_K|^{1/n_K} \to \infty$. Thus, Corollary \ref{Dedekind-zeta-corollary} yields a uniqueness theorem for only finitely many number fields. More interesting application is in the context of asymptotically exact families introduced by Tsfasman-Vl\u{a}du\c{t} \cite{Tsfasman} in 2002.\\

For a number field $K$ and any prime power $q$, let $N_q(K)$ denote the number of non-archimedean places $v$ of $K$ such that $Norm(v)=q$. A sequence $\mathcal{K} = \{K_i\}_{i\in\mathbb{N}}$ of number fields is said to be a family if $K_i\neq K_j$ for $i\neq j$. We say that a family $\mathcal{K}$ is \textit{asymptotically exact} if the limits
\begin{equation*}
\phi_{\mathbb{R}} := \lim_{i\to \infty} \frac{r_1(K_i)}{g_{K_i}}, \hspace{5mm} \phi_{\mathbb{C}} := \lim_{i\to \infty} \frac{r_2(K_i)}{g_{K_i}},\hspace{5mm}
\phi_q := \lim_{i\to\infty} \frac{N_q(K_i)}{g_{K_i}}
\end{equation*}
exist for all prime powers $q$, where $r_1(K_i)$ and $r_2(K_i)$ are the number of real and complex embeddings of $K_i$ respectively. We say that an asymptotically exact family $\mathcal{K} = \{K_i\}$ is \textit{asymptotically bad}, if $\phi_q = \phi_{\mathbb{R}} = \phi_{\mathbb{C}} =0$ for all prime powers $q$. This is analogous to saying that the root discriminant $ |\Delta_{K_i}|^{1/n_{K_i}}$ tends to infinity as $i\to \infty$. If an asymptotically exact family $\mathcal{K}$ is not asymptotically bad, we say that it is \textit{asymptotically good}.\\

Most naturally occurring families of number fields are asymptotically bad families. On the other hand, asymptotically good families are rather mysterious and very little is known about them. In most cases, one must assume the asymptotically good family to be a tower of number fields to prove anything interesting (see for instance \cite[Theorem 7.3]{Tsfasman}). It is important to note that the root discriminant $|\Delta_K|^{1/n_K}$ converges to a non-zero limit over an asymptotically good family. Thus, Corollary \ref{Dedekind-zeta-corollary} becomes interesting in this case and sheds light on how many $\zeta_K$ can have the same degree $n_K$, discriminant $\Delta_K$ and residue $\rho_K$ in an asymptotically good family of number fields.

\section{\bf Acknowledgements}
I am grateful to Prof. M. R. Murty and Prof. V. K. Murty for their suggestions and comments on an earlier version of this paper. This work was supported by a Coleman postdoctoral fellowship at Queen's University, Kingston.

%Let $L(s,f\times g)$ denote the $L$-function associated to the Rankin-Selberg convolution of any two holomorphic newforms $f$ and $g$. Here $f$ and $g$ are normalized Hecke eigenforms of weight $k$. It is known that if $L(s, f\times g)$ has a pole at $s=1$, then $f=g$. Moreover, for $F(s) = L(s,f\times f)$ the degree $d_F = 4$. Hence, from Theorem \ref{main-theorem}, we have the following.

%\begin{corollary}\label{Rankin-Selberg-L-functions}
%For the same value of conductor $q_F$ and residue at $s=1$ $\rho_F$, there are at most $8$ Rankin-Selberg $L$-functions $L(s,f\times f)$ satisfying
%\begin{equation*}
    
%\end{equation*}
%\end{corollary}

\bibliographystyle{abbrv}
\bibliography{extremal}
\end{document}